\documentclass[peerreview,draftcls,onecolumn,12pt]{IEEEtran}

\usepackage{epsfig,psfrag}
\usepackage{amsmath}
\usepackage{mathtools}
\usepackage{amssymb}
\usepackage{here}
\usepackage{mathrsfs}
\usepackage{dsfont}

\newtheorem{theorem}{Theorem}[section]

\newtheorem{lemma}[theorem]{Lemma}

\newtheorem{example}[theorem]{Example}

\DeclareMathOperator*{\col}{col}

\DeclareMathOperator{\He}{He}
\DeclareMathOperator{\eps}{\varepsilon}

\DeclareMathOperator{\co}{co}

\title{Convex dwell-time characterizations for uncertain linear impulsive systems}

\author{Corentin Briat$^1$\thanks{$^1$ The author was previously with the ACCESS Linnaeus Centre, KTH, Stockholm, Sweden. He is now with the Swiss Federal Institute of Technology -- Z\"{u}rich (ETH-Z), Department of Biosystems Science and Engineering (D-BSSE), Mattenstrasse 26, 4058 Basel, Switzerland; email: corentin@briat.info, corentin.briat@bsse.ethz.ch; url: http://www.briat.info} and Alexandre Seuret$^2$\thanks{$^2$ CNRS, Control Systems Department GIPSA-Lab - INP BP 46, Domaine
Universitaire - 38400 Saint Martin d'H\`{e}res - FRANCE; email: alexandre.seuret@gipsa-lab.genoble-inp.fr; url: http://www.gipsa-lab.inpg.fr/\textasciitilde alexandre.seuret/cv\_en.html}}

\begin{document}
\maketitle

\begin{abstract}
New sufficient conditions for the characterization of dwell-times for linear impulsive systems are proposed and are shown to coincide with continuous decrease conditions of a certain class of looped-functionals, a recently introduced type of functionals suitable for the analysis of hybrid systems. This approach allows to consider Lyapunov functions that evolve non-monotonically along the flow of the system in a new way, enlarging then the admissible class of systems which may be analyzed. As a byproduct, due to the particular structure of the obtained conditions, the method is easily extendable to uncertain systems by exploiting some convexity properties. Several examples illustrate the approach.
\end{abstract}

\begin{keywords}
Impulsive systems; dwell-time; looped-functionals; stability; robustness
\end{keywords}

\section{Introduction}

Impulsive systems \cite{Bainov:89, Yang:01b, Cai:05, Cai:08, Hespanha:08, Briat:11l} are an important class of hybrid systems in which the system trajectories admit discontinuities at certain time-instants. They arise in several fields like epidemiology \cite{Stone:00,Briat:09h}, sampled-data and networked control systems \cite{Naghshtabrizi:08}, etc. Among the wide class of impulsive dynamical systems, we may distinguish systems whose impulse-times depend on the system state and those for which they are external and only time-dependent. Linear systems of the latter class may be represented in the following form
\begin{equation}\label{eq:mainsyst}
  \begin{array}{lcl}
    \dot{x}(t)&=&Ax(t),\ t\ne t_k,\ k\in\{1,2,\ldots\}\\
    x^+(t)&=&Jx(t),\ t=t_k,\ k\in\{1,2,\ldots\}\\
    x(t_0)&=&x_0
  \end{array}
\end{equation}
where $x,x_0\in\mathbb{R}^n$ are the system state and the initial condition respectively. The state $x(t)$ is assumed to be left-continuous with $x^+(t):=\lim_{s\downarrow t}x(s)$ and the matrices $A$ and $J$ may be uncertain. The sequence $\{t_1,t_2,t_3,\ldots\}$ is a strictly increasing sequence of impulse times in $(t_0,\infty)$ for some initial time $t_0$. The sequence is assumed to be finite, or infinite and unbounded. The distance between two consecutive impulse times is denoted by $T_k:=t_{k+1}-t_k$ with the additional assumption that $\eps<T_k$, for some $\eps>0$, which excludes the possible existence of a finite accumulation point. Given a set $\mathcal{S}\subseteq(0,\infty)$, we define $\mathbb{I}_{\mathcal{S}}:=\left\{\{t_1,t_2,\ldots\}:\ t_{k+1}-t_k\in \mathcal{S},\ k\in\mathbb{N}\right\}$ to characterize the sequence of impulse times.

According to the matrices $A$ and $J$, the system may exhibit different behaviors. In the case of periodic impulses with period $T_k=T$, determining stability essentially reduces to study the Schurness of the matrix $Je^{AT}$, which is a very simple problem. However, this formulation suffers from several critical drawbacks:
\begin{enumerate}
   \item  the eigenvalue analysis is not extendable to the case of aperiodic impulses since the spectral radius is, in general, not submultiplicative;
   \item discrete-time Lyapunov approaches are thus needed, but they lead to robust LMIs with scalar uncertainties at the exponential, known to be complex numerically, but yet solvable \cite{Oishi:10};
   \item the extension to robust stability analysis is also difficult, again due to the exponential. There is no efficient solution, at that time, to handle matrix uncertainties at the exponential.
\end{enumerate}

The approach discussed in this paper aims at overcoming these problems. The method originates from \cite{Seuret:11b, Seuret:11} where an implicit but equivalent correspondence between discrete- and continuous-time domains is obtained by showing that discrete-time stability is equivalent to a very particular type of continuous-time stability proved using a new type of functionals, referred to as \emph{looped-functionals} \cite{Briat:11l,Briat:12d}. Since discrete-time stability focuses on the decrease of Lyapunov functions along the flow of a system at pointwise instants only, i.e. $V(x(t))$, $t\in\{t_1,t_2,t_3,\ldots\}\cup\{\infty\}$, its continuous-time extension $V(x(t))$, $t\in(0,\infty)$ may then be allowed to be nonmonotonic between consecutive pointwise instants. This hence authorizes the use of continuous-time Lyapunov functions that evolve non-monotonically along the flow of the system, enlarging then the class of admissible impulsive systems which can be possibly analyzed using existing approaches \cite{Hespanha:08,Cai:08}. Some results relying on a specific functional have been obtained for impulsive systems in \cite{Briat:11l,Briat:12c}. Although leading to interesting results in some cases, selecting a functional however introduces some conservatism.

\begin{figure}[h]
  \includegraphics[width=0.45\textwidth]{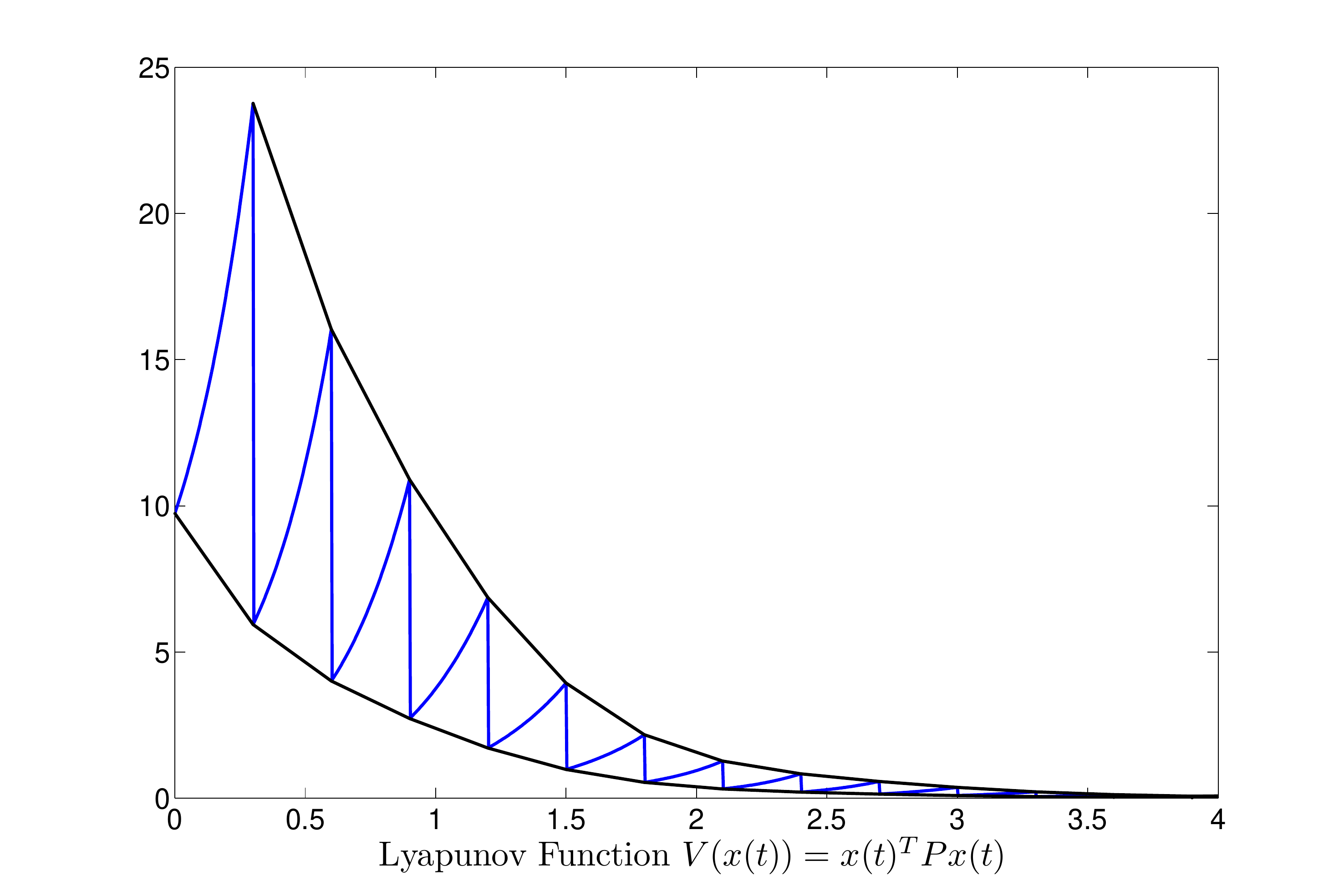}
  \caption{Continuous-time Lyapunov function $V$ (plain) for Example \ref{ex:max} and the discrete-time envelopes (dashed). The discrete-time envelopes are monotonically decreasing while the continuous-time function is highly non-monotonic along the flow of the system.}\label{fig:loc2}
\end{figure}

The current article develops a more generic framework based on an implicit characterization of a subset of the admissible looped-functionals. Since no particular functional is chosen in this paper, the obtained results are in essence more computationally-oriented than those in \cite{Briat:11l,Briat:12c} where functionals are specifically chosen. The obtained conditions are expressed as robust convex optimization problems, conveniently depending on the inter-impulses interval distance and on the system matrices. They are also devoid of exponential terms and can be then efficiently extended to the uncertain case, usually not considered in the literature due to a lack of tractability. Thanks to the discrete-time nature of the stability notion, the approach is able to consider the jumps precisely by authorizing the continuous-time Lyapunov function to be strongly non-monotonic along the flow of the system; see Fig. \ref{fig:loc2}. This feature allows to both characterize stability under periodic and aperiodic impulses. In the aperiodic case, stability under different dwell-times \cite{Liberzon:03, Geromel:06b,Hespanha:08} are considered: minimal dwell-time $T_k\in[T_{min},+\infty)$, maximal dwell-time $T_k\in(0,T_{max}]$ and ranged dwell-time $T_k\in[T_{min},T_{max}]$. The resulting feasibility problems are then solved using polynomial programming, such as sum-of-squares \cite{Parrilo:00} programming, very suitable for this kind of problems.

The paper is structured as follows: Section \ref{sec:DT} is devoted to preliminary results on dwell-times for impulsive systems. In Section \ref{sec:explessDT}, alternative sufficient conditions for dwell-times are provided and applied in Section \ref{sec:robstab} to uncertain systems. Finally, Section \ref{sec:sos} is devoted to the presentation of Sum-of-Squares (SOS) programming used in the examples to enforce the conditions stated in the main results. Illustrative examples are treated in the related sections using SOSTOOLS \cite{sostools} and SeDuMi \cite{Sturm:99}.

The notation is quite standard. The sets of symmetric and positive definite matrices of dimension $n$ are denoted by $\mathbb{S}^n$ and $\mathbb{S}_{+}^n$ respectively. Given two symmetric matrices $A,B$, $A\succ B$ means that $A-B$ is positive definite. For a square matrix $A$, the operator $\He(A)$ stands for the sum $A+A^T$. The identity matrix of dimension $n$ is denoted by $I_n$.

\section{Dwell-time results}\label{sec:DT}

\subsection{Ranged Dwell-Time}

The most general dwell-time notion is referred to as \emph{ranged dwell-time} and is characterized in this section. Minimal and maximal dwell-times are corollaries of this result.
\begin{lemma}[Ranged Dwell-Time]\label{lem:rangedDT}
  Assume there exists a matrix $P\in\mathbb{S}_{+}^n$ such that
  \begin{equation}\label{eq:lolilol23}
    J^Te^{A^T\theta}Pe^{A\theta}J-P\prec0
  \end{equation}
  holds for all $\theta\in[T_{min},T_{max}]$. Then, for any impulse sequence in $\mathbb{I}_{[T_{min},T_{max}]}$, the system (\ref{eq:mainsyst}) is globally asymptotically stable.
\end{lemma}
\begin{proof}
    The condition (\ref{eq:lolilol23}) implies that $V(x)=x^TPx$ is a discrete-time Lyapunov function for the discrete-time system $z(k+1)=e^{AT_k}Jz(k)$, $T_k\in[T_{min},T_{max}]$, since (\ref{eq:lolilol23}) is equivalent to the condition $V(x(t_{k+1}))-V(x(t_k))<0$ for all $t_{k+1}-t_k\in[T_{min},T_{max}]$ and all $x(t_k)\in\mathbb{R}^n$. This implies in turn that $||x(t_k)||_2\to0$ as $k\to\infty$. It remains to prove that the convergence of $||x(t_k)||_2$ to 0 implies the convergence of $||x(t)||_2$ to 0 when $t\to\infty$. Defining $t=t_k+\tau$, we have $x(t)=e^{A\tau}Jx(t_k)$ and hence
\begin{equation*}
  V(x(t))\le\mu ||x(t_k)||_2^2
\end{equation*}
where $\mu=\sup_{s\in[T_{min},T_{max}]}\lambda_{max}(J^Te^{A^T\tau}Pe^{A\tau}J)$. Thus, when $||x(t_k)||_2\to0$ then $||x(t)||_2\to0$. The proof is complete.
\end{proof}
The main difficulty of the above result lies in solving the robust semidefinite feasibility problem (\ref{eq:lolilol23}) over the whole range of $\theta$. The presence of exponential terms also prevents to easily derive a tractable generalization of this result to the uncertain case.

\subsection{Minimal Dwell-Time}

The minimal dwell-time result corresponds to the case when $T_{min}=\bar{T}$, for some $\bar{T}>0$, and $T_{max}\to\infty$. However, since the interval $[\bar{T},+\infty)$ is of infinite measure, it seems rather difficult to computationally check the condition (\ref{eq:lolilol23}) over such an interval. Fortunately, it is possible to provide a finite dimensional alternative result.
\begin{lemma}[Minimal Dwell-Time]\label{lem:minDT}
  Assume that for some given $\bar{T}>0$, there exists a matrix $P\in\mathbb{S}^n_{+}$ such that
  \begin{equation}\label{eq:lolilol1}
    A^TP+PA\prec0
  \end{equation}
  and
  \begin{equation}\label{eq:lolilol2}
    J^Te^{A^T\bar{T}}Pe^{A\bar{T}}J-P\prec0
  \end{equation}
  hold. Then, for any impulse sequence in $\mathbb{I}_{[\bar{T},+\infty)}$, the system (\ref{eq:mainsyst}) is globally asymptotically stable.
\end{lemma}

\begin{proof}
The goal is to show that the conditions of Lemma \ref{lem:minDT} implies that
$J^Te^{A^T\theta}Pe^{A\theta}J-P\prec0$ for all $\theta\in[\bar{T},+\infty)$. A Taylor expansion of $\Phi(\theta):=e^{A^T\theta}Pe^{A\theta}$ around $\theta_0\ge0$ yields
\begin{equation*}
  \Phi(\theta_0+\delta)=\Phi(\theta_0)+e^{A^T\theta_0}(A^TP+PA)e^{A\theta_0}\delta+o(\delta).
\end{equation*}
Hence, condition (\ref{eq:lolilol1}) implies that the map $\Phi(\theta)$ is strictly decreasing, in the sense that $\Phi(\theta_0+\epsilon)\prec\Phi(\theta_0)$, for any $\theta_0>0$ and $\epsilon>0$. Combining this with (\ref{eq:lolilol2}) yields
 \begin{equation*}
    J^Te^{A^T(\bar{T}+\epsilon)}Pe^{A(\bar{T}+\epsilon)}J-P\prec0
 \end{equation*}
 for all $\epsilon\ge0$. The proof of convergence of the state $x(t)$ is finally proved as in the proof of Lemma \ref{lem:rangedDT}.
\end{proof}

\subsection{Maximal Dwell-Time}

The maximal dwell-time\footnote{This is also referred to as \emph{reverse dwell-time} in \cite{Hespanha:08}} result can be obtained by letting $T_{max}=\bar{T}$, for some $\bar{T}>0$, and $T_{min}\to0$ in Lemma \ref{lem:rangedDT}. In this case, the interval $(0,\bar{T}]$ is of finite measure and the problem of infinite measure does not hold. It is hence enough to consider the condition (\ref{eq:lolilol23}) over this interval.

It is however possible, under the restriction that $A$ is anti-Hurwitz, to reduce the condition to a finite-dimensional one, similarly as for the minimal dwell-time.
\begin{lemma}[Maximal Dwell-Time]\label{lem:maxDT}
  Assume that for some given $\bar{T}>0$, there exists a matrix $P\in\mathbb{S}_{+}^n$ such that
 \begin{equation}\label{eq:lolilol11}
    A^TP+PA\succ0
  \end{equation}
  and
  \begin{equation}\label{eq:lolilol22}
    J^Te^{A^T\bar{T}}Pe^{A\bar{T}}J-P\prec0
  \end{equation}
  hold. Then, for any impulse sequence in $\mathbb{I}_{(0,\bar{T}]}$, the system (\ref{eq:mainsyst}) is globally asymptotically stable.
\end{lemma}
\begin{proof}
  The proof is similar to the one of Lemma \ref{lem:minDT}.
\end{proof}

\section{Alternative dwell-times characterization}\label{sec:explessDT}

The main and common drawback of the results of Section \ref{sec:DT} is the presence of exponential terms, preventing any extension to robust dwell-time characterization. Efficiently dealing with matrix uncertainties at the exponential is indeed an open and difficult problem. Checking the condition (\ref{eq:lolilol23}) over a finite-length interval is also computationally demanding, again due to the exponential terms.

We provide here a way for approximating these LMI constraints involving exponential terms by exponential-less ones. These conditions are expressed as robust LMIs containing an additional infinite-dimensional decision variable satisfying a certain boundary condition. It is also emphasized that the results can be interpreted as the continuous-time decrease condition of a very general looped-functional.

\subsection{Ranged Dwell-Time:}

As in the previous section, we start with the most general ranged dwell-time result:
\begin{theorem}\label{th:fund}
 The system (\ref{eq:mainsyst}) is globally asymptotically stable for any impulse sequence in $\mathbb{I}_{[T_{min},T_{max}]}$ if there exist a matrix $P\in\mathbb{S}_+^n$, a scalar $\eps>0$ and a matrix function $Z:[0,T_{max}]\times[T_{min},T_{max}]\to\mathbb{S}^{3n}$, differentiable with respect to the first variable, verifying
        \begin{equation}\label{eq:loliloolldr}
        \begin{array}{rclclcl}
           Y_1^TZ(0,\theta)Y_1-Y_2^TZ(\theta,\theta)Y_2=0
        \end{array}
        \end{equation}
        for all $\theta\in[T_{min},T_{max}]$ where
        \begin{equation*}
          Y_1=\begin{bmatrix}
            J & 0\\
            I_n & 0\\
            0 & I_n
          \end{bmatrix},\ Y_2=\begin{bmatrix}
            0 & I_n\\
            I_n & 0\\
            0 & I_n
          \end{bmatrix}
        \end{equation*}
        and such that the parameter dependent LMI
        \begin{equation}\label{eq:expless}
\Psi(\theta)+\He\left\{\begin{bmatrix}
   A & 0 & 0\\
   0 & 0 & 0\\
   0 & 0 & 0
 \end{bmatrix}^TZ(\tau,\theta)\right\}+\dfrac{\partial}{\partial\tau}Z(\tau,\theta)\preceq0
\end{equation}
holds for all $\tau\in[0,\theta]$ and $\theta\in[T_{min},T_{max}]$ with
\begin{equation}\label{eq:Psi}
   \Psi(\theta):=\begin{bmatrix}
   \theta(A^TP+PA) & 0 & 0\\
   0 & J^TPJ-P+\eps I_n  & 0\\
   0 & 0 & 0
 \end{bmatrix}.
\end{equation}

Moreover, in such a case, the quadratic form $V(x)=x^TPx$ is a discrete-time Lyapunov function for the discrete-time formulation of system (\ref{eq:mainsyst}) and Lemma \ref{lem:rangedDT} holds with the same matrix $P$.
%
%
\end{theorem}
\begin{proof}
  To see the implication of Lemma \ref{lem:rangedDT}, first pre- and post-multiply (\ref{eq:expless}) by $\xi_k(\tau):=\col(x(t_k+\tau),x(t_k),x(t_{k+1}))$, where $x(t_k+\tau)=e^{A\tau}Jx(t_k)$, $\tau\in[0,T_k]$, $t_{k+1}=t_k+T_k$, to obtain
\begin{equation}\label{eq:xit}
\xi_k(\tau)^T\Psi(T_k)\xi_k(\tau)+\dfrac{\partial}{\partial\tau}\left[\xi_k(\tau)^TZ(\tau,T_k)\xi_k(\tau)\right]\le0
\end{equation}
where we have set $\theta=T_k$ since we are considering the interval $[t_k,t_{k+1}]$. The second term of (\ref{eq:xit}) has been obtained using the fact that
\begin{equation}
  \dfrac{d}{d\tau}\xi_k(\tau)=\begin{bmatrix}
    A & 0 & 0\\
    0 & 0 & 0\\
    0 & 0 & 0
  \end{bmatrix}\xi_k(\tau).
\end{equation}
Integrating (\ref{eq:xit}) from $0$ to $T_k$ we get
\begin{equation}\label{eq:dsjqdlmsqdmdr}
\begin{array}{lcl}
  \eta_k&:=&\int_0^{T_k}\left[x(t_k)^T(J^TPJ-P+\eps I_n)x(t_k)\right.\\
  &&+\left.T_k\dfrac{d}{d\tau}V(x(t_k+\tau))\right]d\tau\\
  &&+\xi_k(T_k)^TZ(T_k,T_k)\xi_k(T_k)-\xi_k^+(0)^TZ(0,T_k)\xi_k^+(0)
\end{array}
\end{equation}
where $V(x)=x^TPx$. Noting that
\begin{equation}\label{eq:constraint}
\xi_k^+(0)=Y_1\begin{bmatrix}
  x(t_k)\\
  x(t_{k+1})
\end{bmatrix},\quad \xi_k(T_k)=Y_2\begin{bmatrix}
  x(t_k)\\
  x(t_{k+1})
\end{bmatrix},
\end{equation}
we can see that the constraints (\ref{eq:loliloolldr}) make the two last terms of (\ref{eq:dsjqdlmsqdmdr}) equal to 0, and hence we have
\begin{equation}
  \begin{array}{lcl}
  \eta_k&=&T_k\left[V(x^+(t_k))-V(x(t_k))+\eps||x(t_k)||_2^2\right.\\
  &&\left.+V(x(t_{k+1}))-V(x^+(t_k))\right]\\
  &=&T_k\left[V(x(t_{k+1}))-V(x(t_k))+\eps||x(t_k)||_2^2\right].
\end{array}
\end{equation}
Since (\ref{eq:xit}) is nonpositive, then $\eta_k\le0$ for all $k$ and hence the feasibility of the conditions of Theorem \ref{th:fund} implies the feasibility of the conditions of Lemma \ref{lem:rangedDT}. This concludes the proof.
\end{proof}

The above theorem provides then a sufficient condition for the asymptotic stability of linear impulsive systems for any arbitrary impulse sequence in $\mathbb{I}_{[T_{min},T_{max}]}$. The condition (\ref{eq:expless}) can be understood as the continuous decrease over $\tau\in[0,\theta]$ of the following family of looped-functionals parameterized by $\theta\in[T_{min},T_{max}]$:
\begin{equation}
  W:=x(t_k+\tau)^TPx(t_k+\tau)+\xi_k(\tau)^TZ(\tau,\theta)\xi_k(\tau)+\tau x(t_k)^T(J^TPJ-P+\eps I)x(t_k)
\end{equation}
defined for $\tau\in[0,\theta]$ and satisfying the looping condition 
\begin{equation}\label{eq:loop}
  \xi_k^+(0)^TZ(0,\theta)\xi_k^+(0)=\xi_k(\theta)^TZ(\theta,\theta)\xi_k(\theta).
\end{equation}
The obtained conditions however take the form of an infinite dimensional semi-infinite feasibility problem, which may be hard to solve. When the matrix function $Z(\tau,T)$ is chosen in the set of matrix polynomials, sum-of-squares techniques can be applied in order to solve this problem. This framework allows to easily consider the equality constraint (looping-condition) as an equality constraint on the polynomial coefficients. One of the advantages of the proposed alternative formulation (\ref{eq:expless}) is to provide conditions which are structurally well suited (convex) for the generalization to the uncertain case. Several dwell-time results for uncertain systems are derived in Section \ref{sec:robstab}.


\begin{example}\label{ex:ranged}
    Let us consider the system (\ref{eq:mainsyst}) with matrices \cite{Briat:11l}
  \begin{equation}
  \begin{array}{lclclcl}
        A&=&\begin{bmatrix}
       -1 & 0.1\\
       0 & 1.2
    \end{bmatrix},&& J&=&\begin{bmatrix}
      1.2 &0\\
      0 &0.5
    \end{bmatrix}.
  \end{array}
  \end{equation}
  Since both $A$ and $J$ have unstable eigenvalues, the system is unstable for sufficiently small and large periods, and cannot be studied using many of the existing approaches \cite{Hespanha:08, Cai:08}. In order to determine the lower and upper bounds for the inter-impulse intervals, Theorem \ref{th:fund} is applied and leads to estimates of Table \ref{tab:cranged}. It is important to stress that, in the periodic case, the minimal and maximal periods are given by $0.1824$ and $0.5776$, showing that the same bounds are achieved in the aperiodic case and that the proposed result is nonconservative for a multivariate polynomial $Z(\tau,T)$ of degree 3.
  \begin{table}[h]
      \centering
    \begin{tabular}{|c|cc|}
    \hline
      order of $Z$& $T_{min}$ & $T_{max}$\\
      \hline
      \hline
      1 & 0.2040 & 0.5672\\   
      2 & 0.1824 & 0.5774\\
      3 & 0.1824 & 0.5776\\
      4 & 0.1824 & 0.5776\\
      \hline
    \end{tabular}
    \caption{Estimates of the admissible inter-impulse intervals for the aperiodic system of Example \ref{ex:ranged}}\label{tab:cranged}
  \end{table}
\end{example}

\subsection{Minimal Dwell-time}

The exponential-less approximation of the conditions of Lemma \ref{lem:minDT} is given below:
\begin{theorem}[Minimal Dwell-Time]\label{th:explessmin}
Assume that for some $\bar{T}>0$, there exist a scalar $\eps>0$, a matrix $P\in\mathbb{S}_{+}^n$ and a continuously differentiable matrix function $Z:[0,\bar{T}]\to\mathbb{S}^{3n}$ satisfying $Y_1^TZ(0)Y_1-Y_2^TZ(\bar{T})Y_2=0$ and such that the LMIs
\begin{equation}\label{eq:minDT1}
  A^TP+PA\prec0
\end{equation}
\begin{equation}\label{eq:minDT2}
\begin{bmatrix}
   \bar{T}(A^TP+PA) & 0 & 0\\
   0 & J^TPJ-P+\eps I_n & 0\\
   0 & 0 & 0
 \end{bmatrix}+\He\left\{\begin{bmatrix}
   A & 0 & 0\\
   0 & 0 & 0\\
   0 & 0 & 0
 \end{bmatrix}^TZ(\tau)\right\}+\dfrac{d}{d\tau}Z(\tau)\preceq0
\end{equation}
hold for all $\tau\in[0,\bar{T}]$. Then, for any impulsive sequence in $\mathbb{I}_{[\bar{T},+\infty)}$, the system (\ref{eq:mainsyst}) is globally asymptotically stable and Lemma \ref{lem:minDT} holds with the same matrix $P$.
\end{theorem}
\begin{proof}
  The proof follows from Lemma \ref{lem:minDT} and Theorem \ref{th:fund}.
\end{proof}

\begin{example}\label{ex:min}
  Let us consider the system (\ref{eq:mainsyst}) with matrices \cite{Briat:11l}
  \begin{equation}
  \begin{array}{lclclcl}
        A&=&\begin{bmatrix}
   -1 &0\\
   1 &-2
    \end{bmatrix},& &J&=&\begin{bmatrix}
      2 & 1\\
      1 & 3
    \end{bmatrix}.
  \end{array}
  \end{equation}
  Since $A$ is Hurwitz, the minimal dwell-time concept applies and using Lemma \ref{lem:minDT} we get the value 1.14053 with the matrix $$P=\begin{bmatrix}
   0.6351  &  0.0601\\
    0.0601  &  0.3649
\end{bmatrix}.$$ Considering then Theorem \ref{th:explessmin} with polynomial $Z(\tau)$ of degree 3, we use the toolbox SOSTOOLS \cite{sostools} for sum-of-squares programming to enforce the conditions of Theorem \ref{th:explessmin} and we get the value $1.14054$, which is very close to the one computed with Lemma \ref{lem:minDT}. In such a case, the following matrix is obtained $$P=\begin{bmatrix}
  0.3443 & -0.1102\\
  -0.1102&   0.2285
\end{bmatrix}.$$ Estimates of the minimal-dwell-time for different degrees for $Z$ are given in Table \ref{tab:cmin}.
\end{example}

\begin{table}[ht]
\begin{minipage}[b]{0.5\linewidth}\centering
\begin{tabular}{|c|c|}
\hline
      order of $Z$ & $T_{min}$\\       
      \hline
      \hline
      1 & 4.3255\\
      2 & 1.1407\\
      3 & 1.14054\\
      4 & 1.14054\\
      5 & 1.14054\\
      \hline
    \end{tabular}
    \caption{Estimates of the minimal dwell-time for Example \ref{ex:min}}\label{tab:cmin}
\end{minipage}
\hspace{0.5cm}
\begin{minipage}[b]{0.5\linewidth}
\centering
\begin{tabular}{|c|c|}
\hline
      order of $Z$ & $T_{min}$\\
      \hline
      \hline
      1 &	  36.3071\\   
      2 &	  4.0705\\     
      3 & 	  2.6893\\     
      4 &   2.2966\\     
      5 &	  2.2209\\     
      6 &	  2.2010\\     
      7 &   2.1979\\     
      8 &   2.1974\\     
      9 &   2.1983\\     
      10 & 2.1990\\     
      11 & infeasible\\     
      \hline
    \end{tabular}
        \caption{Estimates of the minimal dwell-time for Example \ref{ex:awful}}\label{fig:awful}
\end{minipage}
\end{table}

\begin{example}\label{ex:awful}
   Let us consider now the system (\ref{eq:mainsyst}) with matrices \cite{Cai:08}
  \begin{equation}\label{eq:awful}
  \begin{array}{lclclcl}
        A&=&\begin{bmatrix}
       -1 & 100\\
       -1 & -1
    \end{bmatrix},&& J&=&\begin{bmatrix}
      0 & -0.9\\
      0.9 & 0
    \end{bmatrix}.
  \end{array}
  \end{equation}
  In the $T$-periodic case, this system exhibits a quite complicated behavior for small $T$'s, alternating stable and unstable regions. Since $A$ is Hurwitz, we can determine the minimal dwell-time using Lemma \ref{lem:minDT} and we get the minimum dwell-time value $2.1254$. Theorem \ref{th:explessmin} leads on the other hand to the estimates of Table \ref{fig:awful}. We can see that, in this case, a gap still persists and the smallest one is obtained with a polynomial $Z(\tau)$ of degree 8. Numerical problems arise when using polynomials of degree larger than 10, resulting then in infeasible problems.
\end{example}

The above examples show some limitations of the approach when the convergence of the sequence of polynomials $Z(\tau)$ of increasing orders, towards a suitable function, is quite poor. In such a case, a high order polynomial would be needed to obtain accurate results, but this would inexorably lead to feasibility problems involving a large number of decision variables and suffering from numerical problems. It is indeed well-known that SOS-programming scales very poorly with the problem complexity, and it is a matter of future research to see how the complexity of the problem addressed in this paper can be reduced.

\subsection{Maximal dwell-time}

A general maximal dwell-time result can be straightforwardly obtained be setting $T_{min}=0$ in Theorem \ref{th:fund}. When the matrix $A$ is anti-Hurwitz, we have the following simpler result, corresponding to Lemma \ref{lem:maxDT}:
\begin{theorem}[Maximal Dwell-Time]\label{th:explessmax}
Assume that for some $\bar{T}>0$, there exist a scalar $\eps>0$, a matrix $P\in\mathbb{S}_{+}^n$ and a continuously differentiable matrix function $Z:[0,\bar{T}]\to\mathbb{S}^{3n}$ satisfying $Y_1^TZ(0)Y_1-Y_2^TZ(\bar{T})Y_2=0$ and such that the LMIs
\begin{equation}
  A^TP+PA\succ0
\end{equation}
\begin{equation}
\begin{bmatrix}
   \bar{T}(A^TP+PA) & 0 & 0\\
   0 & J^TPJ-P+\eps I_n  & 0\\
   0 & 0 & 0
 \end{bmatrix}+\He\left\{\begin{bmatrix}
   A & 0 & 0\\
   0 & 0 & 0\\
   0 & 0 & 0
 \end{bmatrix}^TZ(\tau)\right\}+\dfrac{d}{d\tau}Z(\tau)\preceq0
\end{equation}
hold for all $\tau\in[0,\bar{T}]$. Then, for any impulsive sequence in $\mathbb{I}_{(0,\bar{T}]}$, the system (\ref{eq:mainsyst}) is globally asymptotically stable and Lemma \ref{lem:maxDT} holds with the same matrix $P$.
\end{theorem}
\begin{proof}
  The proof follows from Lemma \ref{lem:maxDT} and Theorem \ref{th:fund}.
\end{proof}

\begin{example}\label{ex:max}
  Let us consider the system (\ref{eq:mainsyst}) with matrices \cite{Briat:11l}
  \begin{equation}
  \begin{array}{lclclcl}
        A&=&\begin{bmatrix}
   1 &3\\
   -1 &2
    \end{bmatrix},& &J&=&\begin{bmatrix}
      1/2 & 0\\
      0 & 1/2
    \end{bmatrix}.
  \end{array}
  \end{equation}
  Since the matrix $A$ is anti-Hurwitz and $J$ is Schur, Theorem \ref{th:explessmax} may be applied and leads to the results of Table \ref{tab:cmax}. For comparison, Lemma \ref{lem:maxDT} yields the value $0.4620$, which is recovered by choosing $Z(\tau)$ as a polynomial of order 3.
\end{example}
\begin{table}[ht]
\begin{minipage}[b]{0.5\linewidth}\centering
\begin{tabular}{|c|c|}
    \hline
      order of $Z$ & $T_{max}$\\
      \hline
      \hline
      1 & 0.3999\\
      2 & 0.4613\\
      3 & 0.4620\\
      4 & 0.4620\\
      \hline
    \end{tabular}
    \caption{Estimates of the maximal dwell-time for Example \ref{ex:max}}\label{tab:cmax}
\end{minipage}
\hspace{0.5cm}
\begin{minipage}[b]{0.5\linewidth}
\centering
\begin{tabular}{|c|c|}
\hline
order of $Z$ & $T_{max}$\\
\hline
\hline
1 & 0.1067\\
2 & 0.1072\\
3 & 0.1072\\
\hline
\end{tabular}
\caption{Estimates of the maximal dwell-time for Example \ref{ex:robustAJ}}\label{fig:robustAJ2}
\end{minipage}
\end{table}

\section{Robust dwell-time analysis}\label{sec:robstab}

In this section, we extend the previous results to the case of uncertain systems with constant polytopic-type uncertainties, that is
\begin{equation}\label{eq:unmat2}
  A\in\mathcal{A}:=\co\left\{A_1,\ldots,A_{N_A}\right\},\quad J\in\mathcal{J}:=\co\left\{J_1,\ldots,J_{N_J}\right\}
\end{equation}
for some integers $N_A,N_J>0$. The system matrices can hence be alternatively written as $$A=\sum_{i=1}^{N_A}\sigma_{A,i}A_i,\ \mathrm{and}\ J=\sum_{i=1}^{N_J}\sigma_{J,i}J_i$$ where both $\sigma_{A}:=\col_i(\sigma_{A,i})$ and $\sigma_{J}:=\col_i(\sigma_{J,i})$ belong to the unit simplex. The exponential terms in the usual dwell-time conditions have been major obstacles to the extension of the results to uncertain systems. Thanks to the proposed alternative formulations, considering uncertain systems is henceforth possible.
The theorem below provides a result concerning the characterization of ranged dwell-time for uncertain systems:
\begin{theorem}\label{th:fund3}
 Assume there exist a matrix $P\in\mathbb{S}_{+}^n$, a scalar $\eps>0$, $N_J$ matrix functions $Z_i:[0,T_{max}]\times[T_{min},T_{max}]\to\mathbb{S}^{3n}$, $i=1,\ldots,N_J$, differentiable with respect to the first variable, verifying the condition $(Y_1^i)^TZ_i(0,\theta)Y_1^i-Y_2^TZ_i(\theta,\theta)Y_2=0$ for all $\theta\in[T_{min},T_{max}]$, $i=1,\ldots,N_J$ where
        \begin{equation*}
          Y_1^i=\begin{bmatrix}
            J_i & 0\\
            I_n & 0\\
            0 & I_n
          \end{bmatrix},\ Y_2=\begin{bmatrix}
            0 & I_n\\
            I_n & 0\\
            0 & I_n
          \end{bmatrix}
        \end{equation*}
        and such that the parameter dependent LMIs
        \begin{equation}\label{eq:expless_AJ}
\begin{bmatrix}
   \theta(A_j^TP+PA_j) & 0 & 0\\
   0 & J_i^TPJ_i-P+\eps I_n  & 0\\
   0 & 0 & 0
 \end{bmatrix}+\He\left\{\begin{bmatrix}
   A_j & 0 & 0\\
   0 & 0 & 0\\
   0 & 0 & 0
 \end{bmatrix}^TZ_i(\tau,\theta)\right\}+\dfrac{\partial}{\partial\tau}Z_i(\tau,\theta)\preceq0
\end{equation}
hold for all $\tau\in[0,T]$, $T\in[T_{min},T_{max}]$ and $i=1,\ldots,N_J$ and $j=1,\ldots,N_A$.

Then, the system (\ref{eq:mainsyst})-(\ref{eq:unmat2}) is asymptotically stable for any impulse sequence in $\mathbb{I}_{[T_{min},T_{max}]}$ and the following equivalent statements hold:
  \begin{enumerate}
    \item The quadratic form $V(x)=x^TPx$ is a discrete-time Lyapunov function for the uncertain time-varying discrete-time system $z(k+1)=e^{AT_k}Jz(k)$ for any $T_k\in[T_{min},T_{max}]$ and any $(A,J)\in\mathcal{A}\times\mathcal{J}$.
    \item The LMI
    \begin{equation}
      J^Te^{A^T\theta}Pe^{A\theta}J-P\prec0
    \end{equation}
    holds for any $\theta\in[T_{min},T_{max}]$ and any $(A,J)\in\mathcal{A}\times\mathcal{J}$.
  \end{enumerate}
\end{theorem}

\begin{proof}
  Multiplying (\ref{eq:expless_AJ}) by $\sigma_{A,j}$, summing over $j=1,\ldots,N_A$ and integrating over $[0,T]$ leads to the equivalent LMI $J_i^Te^{A^TT}Pe^{AT}J_i-P\prec0$. Since the LMI holds for all $i=1,\ldots, N_J$, then using a Schur complement, it is possible to show that this is equivalent to $J^Te^{A^TT}Pe^{AT}J-P\prec0$, for all $(A,J)\in\mathcal{A}\times\mathcal{J}$. The proof is complete.

\end{proof}

Let us illustrate the above result with the following example:
\begin{example}\label{ex:robustAJ}
  Let us consider the polytopes
  \begin{equation}
    \mathcal{A}:=\co\left\{\begin{bmatrix}
      1& 3\\
      -1 &2
    \end{bmatrix}, \begin{bmatrix}
      3 &1\\-2& 4
    \end{bmatrix}\right\},\ \mathcal{J}:=\co\left\{\begin{bmatrix}
      1/4 &0\\
      0 &2/3
    \end{bmatrix}, \begin{bmatrix}
      2/3 &0\\0 &1/2
    \end{bmatrix}\right\}.
  \end{equation}
This uncertain system satisfies the maximal dwell-time conditions since the matrices in $\mathcal{A}$ are anti-Hurwitz. We apply Theorem \ref{th:fund3} with $Z_i(\tau)$ and adding the conditions $A_j^TP+PA_j\succ0$, $j=1,\ldots,N_A$. The computed dwell-time estimates for different degrees for the polynomial $Z_i$ are gathered in Table \ref{fig:robustAJ2}. For comparison, the conditions of Lemma \ref{lem:maxDT} are gridded and their solving yields the value $0.1072$, which shows that the proposed approach gives the exact value for a polynomial of degree 2. Note, however, that solving the set of gridded LMIs is much longer than solving the proposed result. Using gridded LMIs is also less accurate since the LMIs are not checked over the whole set $\mathcal{A}$.
\end{example}

\section{General presentation of Sum of Squares as Algorithmic Tool}\label{sec:sos}

The methodology we use to implement the conditions of Theorems
\ref{th:fund}, \ref{th:explessmin}, \ref{th:explessmax} and \ref{th:fund3} is based on the sum-of-squares (SOS) decomposition of positive polynomials. When applying this methodology,
we assume that all matrix functions are polynomial, can be
approximated by polynomials, or there is a change of coordinates that
renders them polynomial.

Denote by $\mathbb R[y]$ the ring of polynomials in $y = (y_1, \dots ,
y_n)$ with real coefficients. Denote by $\Sigma_s$ the cone of
polynomials that admits a SOS decomposition, i.e., those $p
\in\mathbb R[y]$ for which there exist $h_i \in \mathbb R[y], i =
1,\dots,M$ so that $$p(y) = \sum ^M_{ i=1} h^2_i (y).$$

If $p(y) \in\Sigma_s$, then clearly $p(y) \geq 0$ for all $y$. The
converse is however not true, except in some very specific cases. The advantage of SOS is that
the problem of testing whether a polynomial is SOS is equivalent to solving a semidefinite program \cite{Parrilo:00}, hence worst-case polynomial-time
verifiable. Note that on the other hand, checking whether a polynomial is positive is NP-hard. When extended to matrix polynomials, positivity is substituted by \emph{semidefinite positivity}. The semidefinite programs related to SOS can be formulated efficiently
and the solution can be retrieved using
SOSTOOLS \cite{sostools}, which interfaces with semidefinite
solvers such as SeDuMi \cite{Sturm:99}.

Consider now the conditions in Theorem \ref{th:explessmin} which take the form
\begin{equation}\label{MIneq}
L(\tau) \leq 0,\quad  \tau\in\mathcal{S}:=[0,\bar{T}],
\end{equation}
where $L(\tau)\in \mathbb{S}^{3n}$ and $\mathcal S$ is a semialgebraic set
described by polynomial inequalities:
\[ \mathcal S = \{s\in\mathbb{R}:\ g_i(s) \geq 0,\ i =
1,\dots,M\},
\]
where the $g_i$'s are polynomial functions. In order to test condition \eqref{MIneq}, we can
apply Positivstellensatz results, such as the one in \cite{Putinar:93}, which allow to test positivity on a semialgebraic set using SOS programming. Specifically, Condition~\eqref{MIneq} holds if there exist SOS polynomials $P_i(\tau, y)$, such that
\[
L(\tau) + \sum^M_{ i=1} g_i(\tau)P_i(\tau, y) = P_0(\tau).
\]
Intuitively, the above condition guarantees that when $\tau\in\mathcal
S$, we have $L(\tau)\leq -\sum_{i=1}^M g_i(\tau)p_i(\tau, y) \leq 0$ since
$g_i \geq 0$ and $p_i \geq 0$, and therefore $L(\tau) \leq 0$ for those
$\tau$.

\section{Conclusion}

Alternative sufficient conditions for the dwell-time characterization of linear impulsive systems have been developed. They can be interpreted as continuous-time decrease conditions of a certain class of looped-functionals. Ranged, minimal and maximal dwell-times have been considered. The obtained conditions are expressed as robust semidefinite programming problems which may be solved using polynomial techniques, such as sum-of-squares. Thanks to the convenient structure of the conditions, the results have been extended to uncertain systems. Several examples illustrate the efficiency of the approach. Future works will be devoted to the derivation of necessary and sufficient alternative conditions, stabilization criteria and the extension to nonlinear systems.

\section{Acknowledgments}

This work has been supported by the ACCESS (http://www.access.kth.se) and RICSNET projects, KTH, Stockholm.

\bibliographystyle{ieeetran}


\end{document}